\documentclass[11pt,a4paper,reqno]{amsart}
\usepackage[a4paper, left = 2cm,  hmargin={25mm,25mm},vmargin={25mm,25mm}]{geometry}
\usepackage[utf8]{inputenc}
\usepackage{geometry}
\usepackage[english]{babel}
\usepackage{graphicx}
\usepackage{color}
\usepackage{pdfpages}
\usepackage{amsbsy}
\usepackage{amssymb}
\usepackage{amsmath}
\usepackage{mathtools}
\usepackage{matlab-prettifier}
\usepackage{amsthm}
\usepackage{tikz-cd}
\usepackage{amsfonts}
\usepackage{array}
\usepackage{hyperref}
\usepackage[OT2,T1]{fontenc}

\DeclareSymbolFont{cyrletters}{OT2}{wncyr}{m}{n}
\DeclareMathSymbol{\Sha}{\mathalpha}{cyrletters}{"58}
\newcolumntype{P}[1]{>{\centering\arraybackslash}p{#1}}
\everymath{\displaystyle}
\newtheorem{theorem}{Theorem}[section]

\theoremstyle{definition}
\newtheorem{lemma}{Lemma}[section]

\theoremstyle{definition}
\newtheorem{defn}{Definition}[section]

\theoremstyle{definition}
\newtheorem{proposition}{Proposition}[section]

\newtheorem*{notation}{Notation}

\definecolor{lightgray}{gray}{0.5}
\newcommand\m[1]{\begin{pmatrix}#1\end{pmatrix}} 

\usepackage[backend=biber, maxnames=10]{biblatex}
\begin{document}

\title[A note on the twisted degree $6$ $L$-function for Hermitian cusp forms of degree $2$] 
{A note on the twisted degree $6$ $L$-function for Hermitian cusp forms of degree $2$}
\author{Rafail Psyroukis}
\address{Department of Mathematical Sciences\\ Durham University\\
South. Rd.\\ Durham, DH1 3LE, U.K.}
\email{rafail.psyroukis@durham.ac.uk}
\maketitle
\vspace{-1cm}
\begin{abstract}
Let $F$ be a cuspidal Hermitian eigenform of degree two over $\mathbb{Q}(i)$, with first Fourier-Jacobi coefficient not identically zero. Building on a paper by Das and Jha, we prove the meromorphic continuation to $\mathbb{C}$ and the functional equation of a degree six $L$-function attached to $F$ by Gritsenko, twisted by a Dirichlet character.
\end{abstract}
\section{Introduction}

Andrianov initiated the study of the spinor $L$-function attached to a cuspidal Siegel eigenform of degree two. In \cite{andrianov2}, he proved its analytic continuation to the complex plane and its functional equation.\\

A major simplification of the proof of the analytic properties of the spinor $L$-function was given by Kohnen and Skoruppa, in their paper \cite{kohnen_skoruppa}. In particular, the authors considered the Dirichlet series
\begin{equation}\label{dirichlet series}
    D_{F,G}(s) := \zeta(2s-2k+4)\sum_{N=1}^{\infty} \langle \phi_N, \psi_N \rangle N^{-s},
\end{equation}
where $\phi_N, \psi_N$ are the $N$th Fourier-Jacobi coefficients of two degree two Siegel cusp forms $F,G$ of integral weight $k\geq0$ and $\langle \;, \;\rangle$ denotes the Petersson inner product on the space of Jacobi forms of index $N$ and weight $k$.\\

They prove two main results: Firstly, this series admits a meromorphic continuation to $\mathbb{C}$ and satisfies a functional equation. Secondly, when $F$ is a Hecke eigenform and $G$ is in the Maass space, this series is proportional to the spinor $L$-function attached to $F$. The analytic properties of the spinor $L$-function then follow, provided the first Fourier-Jacobi coefficient of $F$ is not identically zero. It turns out that this is not a restriction, as Manickam has shown in \cite{manickam}.\\

Gritsenko, in his paper \cite{gritsenko_zeta}, considered a certain degree six $L$-function attached to a cuspidal Hermitian eigenform of degree two over $\mathbb{Q}(i)$. By using methods similar to Andrianov, he managed to show its meromorphic continuation to $\mathbb{C}$ as well as to establish a functional equation. Moreover, in his later paper \cite{gritsenko}, he considered a same-type Dirichlet series as \eqref{dirichlet series} and proved completely analogous results to the ones of Kohnen and Skoruppa for the Hermitian case, with the same restriction on the first Fourier-Jacobi coefficient. \\

It should be noted here that the analytic properties of the Hermitian analogue of \eqref{dirichlet series} were also proved by Raghavan and Sengupta in \cite{raghavan}. The authors, however, did not obtain the relation of the Dirichlet series with Gritsenko's $L$-function in the general case, but only when both $F,G$ belong in the Maass space.\\

In another direction, twists of \eqref{dirichlet series} by Dirichlet characters were considered, initially by Kohnen in \cite{kohnen} for the degree two case (along with some restrictions on the character) and later by Kohnen, Sengupta, and Krieg in \cite{kohnen_krieg_sengupta} for the arbitrary degree case. The authors proved the analytic properties of these twisted Dirichlet series and, through these, obtained the analytic properties of the twisted spinor $L$-function of Andrianov.\\

Recently, Das and Jha in \cite{das_jha}, proved the analytic properties of a twisted Dirichlet series, analogous to \eqref{dirichlet series}, attached to two cuspidal Hermitian eigenforms of degree two over $\mathbb{Q}(i)$ (here the Dirichlet character is defined on $\mathbb{Z}[i]$).\\

In this note, we combine the results of Das and Jha in \cite{das_jha} with the methods of Gritsenko in \cite{gritsenko}, in order to obtain the analytic properties of the twisted degree six $L$-function attached to a Hermitian cuspidal eigenform of degree two over $\mathbb{Q}(i)$, with first Fourier-Jacobi coefficient not identically zero. Our main contribution is finding the correct zeta and gamma factors in order to define a \textbf{completed} twisted $L$-function and appropriately generalise \cite[Theorem 4.1, 5.1]{gritsenko}.\\

The structure of the paper is as follows: In Section $2$, we give the main definitions for Hermitian modular forms and their Fourier-Jacobi expansions. In Section $3$, we present the relevant Hecke theory we need. Finally, in Section $4$, we give the main Proposition \ref{theorem1}, from which the main Theorem \ref{theorem2} follows, proving the analytic properties of the twisted degree six $L$-function considered.
\begin{notation}
    We denote the space of $n\times n$ matrices with coefficients in a ring $R$ with $\textup{M}_{n}(R)$. By $1_n$, we denote the identity matrix. For any matrix $M \in \textup{Mat}_{n}(R)$, we denote by $\det(M), \textup{ tr}(M)$ the determinant and trace of $M$ respectively. By $\textup{GL}_n(R)$, we denote the matrices in $M_{n}(R)$ with non-zero determinant. For any vector $v$, we denote by $v^{t}$ its transpose. For a complex number $z$ let $\overline{z}$ denote its complex conjugate and $e(z) := e^{2\pi i z}$. We also use the bracket notation $A[B] := \overline{B}^{t}AB$ for suitably sized complex matrices $A,B$. Finally, we write $A \geq B$ for two matrices $A,B$ if $A-B$ is semi-positive definite.
\end{notation}
\section{Preliminaries}
Let $K := \mathbb{Q}(i)$ denote the Gaussian field and $\mathcal{O}_K = \mathbb{Z}[i]$ denote its ring of integers.
\begin{defn}\label{unitary}
Let $R$ be either $K$, $\mathcal{O}_K$ or $\mathbb{C}$ and fix an embedding $R \hookrightarrow \mathbb{C}$. We write $U(2,2)(R)$ for the $R$-points of the \textbf{unitary group} of degree two. That is,
\begin{equation*}
    U(2,2)(R) := \left\{g \in \textup{GL}_{4}(R) \mid J_2[g] = J_2\right\},
\end{equation*}
where $J_2 := \m{0_2 & -1_2\\1_2 &0_2}$.
\end{defn}
Hence, for an element $\m{A&B\\C&D} \in U(2,2)(R)$ with $2\times 2$ matrices $A,B,C,D$, these satisfy the relations
\begin{equation*}
    \overline{A}^{t}C = \overline{C}^{t}A, \textup{ }\overline{D}^{t}B = \overline{B}^{t}D, \textup{ }A\overline{D}^{t} - \overline{B}^{t}C = 1_2.
\end{equation*}

\begin{defn}
The \textbf{Hermitian upper half-plane} of degree two is defined by 
\begin{equation*}
    \mathbb{H}_2 := \{Z \in \textup{M}_2(\mathbb{C})|-i (Z-\overline{Z}^{t})>0\}.
\end{equation*}
We also denote the usual upper half plane by $\mathbb{H}$.
\end{defn}
We fix an embedding $K \hookrightarrow \mathbb{C}$. Then, an element $g = \m{A&B\\C&D} \in U(2,2)(K) \hookrightarrow U(2,2)(\mathbb{C})$ of the unitary group acts on the above upper half plane via the action
\begin{equation*}
    Z \longmapsto g\langle Z\rangle := (AZ+B)(CZ+D)^{-1}.
\end{equation*}

The usual \textbf{factor of automorphy} is defined by $j(g, Z) := \textup{det}(CZ+D)$.\\

Let now $\Gamma_2$ denote the \textbf{Hermitian modular group}, that is $\Gamma_2 := U(2,2)(\mathcal{O}_K)$. We define the following slash operator.
\begin{defn}
    Let $k\in\mathbb{Z}$. Then, for a function $F$ on $\mathbb{H}_{2}$ and a matrix $g = \m{A&B\\C&D} \in U(2,2)(K)$, we define
    \begin{equation*}
        (F\mid_k g)(Z) := \textup{det}(CZ+D)^{-k}F(g\langle Z\rangle).
    \end{equation*}
\end{defn}
\begin{defn}\label{hermitian_modular_form}
A function $F : \mathbb{H}_2 \longrightarrow \mathbb{C}$ is called a \textbf{Hermitian modular form} of degree two and integer weight $k \geq 0$ if
\begin{itemize}
    \item $F$ is holomorphic,
    \item $F$ satisfies 
    \begin{equation*}
        (F|_{k}g)(Z) = F(Z),
    \end{equation*}
for all $g \in \Gamma_2$ and $Z \in \mathbb{H}_2$. 
\end{itemize}
\end{defn}
It is well-known (\cite[Chapter III.2]{Krieg_Book}) that the set of all such forms constitutes a finite-dimensional complex vector space, which we denote by $M_{2}^{k}$. Each such $F$ admits a Fourier expansion
\begin{equation}\label{fourier-expansion}
    F(Z) = \sum_{N \in \Lambda_2, N \geq 0}a(N)e(\textup{tr}(NZ)),
\end{equation}
where $\Lambda_2 := \left\{N = (n_{ij})_{i,j=1}^{2} \mid \overline{N}^{t} = N,  n_{ii} \in \mathbb{Z}, 2n_{ij} \in \mathcal{O}_K \textup{ for } i\neq j\right\}$.\\

$F$ is called a \textbf{cusp form} if $a(N) \neq 0$ only for $N$ positive definite. We denote the space of cusp forms by $S_{2}^{k}$. Finally, we have a notion of a Petersson inner product.
\begin{defn}\label{hermitian_inner_product}
The \textbf{Petersson inner product} for two Hermitian modular forms $F,G$, where at least one of them is a cusp form, is given by
\begin{equation*}
    \langle F, G\rangle := \int_{\Gamma_2\backslash \mathbb{H}_2}{F(Z)\overline{G(Z)}(\det Y)^{k}\textup{d}^{*}Z},
\end{equation*}
where $\textup{d}^{*}Z = (\det Y)^{-2n}\textup{d}X\textup{d}Y$ with $Z=X+iY$.
\end{defn}

Let now $F \in S_2^k$ and $Z \in \mathbb{H}_2$. If we partition $Z = \m{\tau&z_1 \\ z_2 &\omega}$, with $\tau, \omega \in \mathbb{H}$ and $z_1,z_2 \in \mathbb{C}$, we can write the expansion of $F$ with respect to the variable $\omega$ as
\begin{equation}\label{fourier-jacobi expansion}
    F(Z) = \sum_{m=1}^{\infty} \phi_{m}(\tau, z_1, z_2)e(m\omega),
\end{equation}
where $\phi_m : \mathbb{H}\times \mathbb{C}^2 \longrightarrow \mathbb{C}$ are holomorphic functions, called the \textbf{Fourier-Jacobi coefficients} of $F$. These are a special kind of modular forms, with respect to the action of the so-called \textbf{parabolic subgroups}. Let us make this more precise. Let
\begin{equation}\label{parabolic}
    \Gamma_{1,1} := \left\{\m{*&0&*&*\\ *&*&*&*\\ *&0&*&*\\0&0&0&*} \in \Gamma_{2}\right\},
\end{equation}
the Klingen parabolic of $\Gamma_2$. We then have the following Definition.
\begin{defn}\label{parabolic_modular_forms}
A holomorphic function $F$ on $\mathbb{H}_{2}$ is a modular form of weight $k$ with respect to the parabolic subgroup $\Gamma_{1,1}$ if the following conditions hold:
\begin{itemize}
    \item $(F \mid_k M)(Z) = F(Z)$ for all $M \in \Gamma_{1,1}, Z \in \mathbb{H}_2$,
    \item The function $F(Z)$ is bounded in the domain $\textup{Im}(Z)\geq c1_2$, for all $c>0$.
\end{itemize}
\end{defn}
Again, each such $F$ has a Fourier expansion as in equation \eqref{fourier-expansion}, and we call $F$ a \textbf{cusp form} if $a(N) \neq 0$ only for positive definite matrices $N$. Finally, we can define the (Hermitian) Fourier-Jacobi forms as Gritsenko does in \cite[p. 2887]{gritsenko}.
\begin{defn}\label{jacobi}
    A complex-valued, holomorphic function $\phi$ on $\mathbb{H} \times \mathbb{C}^{2} $ is said to be a Fourier-Jacobi form of degree two, weight $k$ and index $m$ if the function 
    \begin{equation*}
    \widetilde{\phi}\left(\m{\tau&z_1\\z_2&\omega}\right):= \phi(\tau, z_1, z_2)e(m\omega)
    \end{equation*}
    is a modular form with respect to the group $\Gamma_{1,1}$, where $\omega \in \mathbb{H}$ is chosen so that $\m{\tau&z_1\\z_2&\omega} \in \mathbb{H}_{2}$. The space of such forms is denoted by $J_{k,m}$ and the corresponding space of cusp forms by $J_{k,m}^{\textup{cusp}}$.
\end{defn}
In particular, the Fourier-Jacobi coefficients of $F$ in \eqref{fourier-jacobi expansion} are elements of $J^{\textup{cusp}}_{k,m}$. Finally, we define the following inner product on $J^{\textup{cusp}}_{k,m}$.
\begin{defn}\label{inner_product_jacobi}
The Petersson inner product of two Fourier-Jacobi forms $\phi, \psi \in J^{\textup{cusp}}_{k,m}$ is defined as
\begin{equation*}
    \langle \phi, \psi \rangle := \int_{\mathcal{F}^{J}}\phi(\tau, z_1,z_2)\overline{\psi(\tau,z_1,z_2)}v^{k}e^{-\pi m |z_1-\overline{z_2}|^2/v}\textup{d}\mu,
\end{equation*}
where $\textup{d}\mu = v^{-4}\textup{d}u\textup{d}v\textup{d}x_1\textup{d}y_1\textup{d}x_2\textup{d}y_2$ with $\tau = u+iv$, $z_j=x_j+iy_j$ for $j=1,2$, and $\mathcal{F}^{J}$ is a fundamental domain for the action of $\Gamma_{1,1}$ on $\mathbb{H} \times \mathbb{C}^{2}$. 
\end{defn}
\section{Hecke Rings and the degree 6 $L$-function}
In this Section, we give an account of the Hermitian Hecke theory we will need. We follow Gritsenko in \cite{gritsenko}. We start by defining the groups of similitude:
\begin{equation*}
    S^{2} := \{g \in M_{4}(K) \mid J_2[g] = \mu(g)J_2, \textup{ for some } \mu(g) > 0\},
\end{equation*}
\begin{equation*}
    S^{2}_{p} := \{g \in S^{2} \cap M_{4}(\mathcal{O}_K[p^{-1}]) \mid \mu(g) = p^{\delta}, \delta \in \mathbb{Z}\},
\end{equation*}
for each rational prime $p$. It is then well-known that the pairs $(\Gamma_2, S^{2}), (\Gamma_{2}, S^{2}_p)$ are Hecke pairs and we can define the corresponding Hecke rings, which we will denote by $H^2$ and $H_{p}^{2}$ respectively (see \cite[pp. 2869, 2870]{gritsenko3} for definitions of Hecke pairs and Hecke rings). From \cite[Corollary $2.2$]{gritsenko}, we can decompose the global Hecke ring into the tensor product of $p$-rings as follows:
\begin{equation*}
    H(\Gamma_2, S^{2}) = \bigotimes_{p}H(\Gamma_2, S^{2}_p).
\end{equation*}
We now want to define a Hecke ring which acts on the Fourier-Jacobi forms. To this end, we define the so-called parabolic Hecke rings. Let $S^{1,1}$ and $S_{p}^{1,1}$ denote the intersection of the groups $S^{2}$ and $S_{p}^{2}$ with $\Gamma_{1,1}$ of equation \eqref{parabolic} respectively. Again, the pairs $(\Gamma_{1,1}, S^{1,1})$ and $(\Gamma_{1,1}, S_{p}^{1,1})$ are Hecke pairs (cf. \cite[Section 3]{gritsenko}) and we can then define the Hecke rings
\begin{equation*}
    H^{1,1} := H(\Gamma_{1,1}, S^{1,1}), \textup{ }H_{p}^{1,1} := H(\Gamma_{1,1}, S_{p}^{1,1}).
\end{equation*}
We can now view the original Hecke rings $H^2$ and $H_p^2$ as subrings of the corresponding parabolic Hecke rings, by virtue of the following Lemma.
\begin{lemma}\label{embeddings}
Let $(\Gamma_0, S_0)$ and $(\Gamma, S)$ be two Hecke pairs. We assume that
\begin{equation*}
    \Gamma_0 \subset \Gamma, \,\,\,\Gamma S_0 = S,\, \,\,\Gamma \cap S_0S_0^{-1} \subset \Gamma_0.
\end{equation*}
Then, given an arbitrary element $X \in H(\Gamma, S)$, according to the second condition, we can write it as
\begin{equation*}
    X = \sum_{i} a_i (\Gamma g_i),
\end{equation*}
with $g_i \in S_0$. Then, if we set
\begin{equation*}
    \epsilon(X) := \sum_{i} a_i (\Gamma_{0}g_i),
\end{equation*}
then $\epsilon$ does not depend on the selection of the elements $g_i \in S_0$ and is an embedding (as a ring homomorphism) of the Hecke algebra $H(\Gamma, S)$ to $H(\Gamma_0, S_0)$.
\end{lemma}
\begin{proof}
See \cite[page 2890]{gritsenko}.
\end{proof}
Now, for each prime $p$, since $\Gamma_{2}S_{p}^{1,1} = S_{p}^{2}$, after writing an element $X \in H_{p}^{2}$ as
\begin{equation*}
    X = \sum_{i} a_i \Gamma_{2} g_i,
\end{equation*}
with $g_i \in S_{p}^{2}$, we can define an embedding $H_p^2 \xhookrightarrow{} H_p^{1,1}$ given by
\begin{equation*}
    X \longmapsto \epsilon(X) = \sum_{i} a_i\Gamma_{1,1}g_i,
\end{equation*}
using Lemma \ref{embeddings}.\\

Moreover, we can embed the classical Hecke algebra $H(\textup{SL}_2(\mathbb{Z}),\textup{GL}_2(\mathbb{Q})) \xhookrightarrow{} H(\Gamma_{1,1}, S^{1,1})$ in two ways, as follows: If $X = \textup{SL}_2(\mathbb{Z})g\textup{SL}_2(\mathbb{Z})$ with $g = \textup{diag}(a, d)$, where $a,d \in \mathbb{Z}$, we define
\begin{equation}\label{+- embeddings}
    j_{-}(X) := \Gamma_{1,1}\textup{diag}(a,\det(g), d, 1)\Gamma_{1,1}, \textup{ }j_{+}(X) := \Gamma_{1,1}\textup{diag}(a,1, d, \det(g))\Gamma_{1,1}.
\end{equation}
These are related by $(j_{-}(X))^{*} = j_{+}(X)$ for all $X \in H(\textup{SL}_2(\mathbb{Z}),\textup{GL}_2(\mathbb{Q}))$, where $* : H^{1,1} \longrightarrow H^{1,1}$ is the anti-homomorphism defined by
\begin{equation}\label{antihomomorphism}
    \sum_{i}a_i\Gamma_{1,1}M_i\Gamma_{1,1} \longmapsto \sum_{i}a_i\Gamma_{1,1}\mu(M_i)M_i^{-1}\Gamma_{1,1},
\end{equation}
with $\mu(M)$ the similitude of $M$. \\

Let us now define how the elements of the parabolic Hecke algebras act on Fourier-Jacobi forms.
Let  
\begin{equation*}
 X = \Gamma_{1,1}\m{*&0&*&*\\**&a&*&*\\**&0&*&*\\0&0&0&b}\Gamma_{1,1} = \sum_{i} \Gamma_{1,1}g_i \in H^{1,1},
\end{equation*}
for some $g_i \in S^{1,1}$. Now, if $F$ is as in Definition \ref{parabolic_modular_forms}, we define
\begin{equation}\label{action on modular forms}
    (F\mid_{k} X)(Z) := \sum_{i}(F\mid_{k} g_i)(Z),
\end{equation}
where for any $g \in S^{1,1}$, we define the slash operator in this case by
\begin{equation*}
    (F \mid_k g)(Z) := \mu(g)^{2k-4}j(g,Z)^{-k}F(g\langle Z\rangle).
\end{equation*}
Gritsenko gave the following very convenient definition of the signature.
\begin{defn}\label{signature}
    The signature of $X$ is defined as $s(X) := b/a$.
\end{defn}
Using the signature $s:=s(X)$ of $X$, we can now define its action on Fourier-Jacobi forms.
\begin{proposition}\label{fourier_jacobi_action}
    Let $\phi \in J_{k,m}$ denote a Fourier-Jacobi form of weight $k$ and index $m$. Then, for $Z = \m{\tau&z_1\\z_2&\omega} \in \mathbb{H}_2$, we define the action of $X$ on $\phi$ via
    \begin{equation*}
        \left(\phi \mid_k X\right)(\tau, z_1, z_2) := \left(\widetilde{\phi} \mid_k X \right)(Z)e\left(-\frac{m}{s}\omega\right),
    \end{equation*}
with $\widetilde{\phi}(Z):= \phi(\tau, z_1, z_2)e(m\omega)$.
Then $\phi \mid_k X$ belongs to $J_{k,m/s}$ if $m/s$ is an integer and is $0$ otherwise.
\end{proposition}
\begin{proof}
    See \cite[Lemma $4.1$]{gritsenko}.
\end{proof}
Finally, we define the degree $6$ $L$-function associated to a Hermitian modular form of degree two. For each prime $p$, we define the following standard elements of $H(\Gamma_2, S_p^2)$:
\begin{itemize}
    \item $T_p := \Gamma_2\textup{diag}(1,1,p,p)\Gamma_2$.
    \item $T_{1,p} := \Gamma_2\textup{diag}(1,p,p^2,p)\Gamma_2$.
    \item $T_{\pi} := \Gamma_2\textup{diag}(1, \pi, p, \pi)\Gamma_2, \textup{ }T_{\overline{\pi}} := \Gamma_2(1, \overline{\pi}, p, \overline{\pi})\Gamma_2$ (if $p = \pi\overline{\pi}$, including the case $p=2$).
    \item $\Delta_q = \Gamma_2\textup{diag}(q,q,q,q)\Gamma_2$, for $q \in \{p, \pi, \overline{\pi}\}$, depending on the prime $p$.
\end{itemize}
Let now $F \in S_2^k$ be a Hecke eigenform for $H^2$, i.e. an eigenfunction for all Hecke elements in $H^2$. We use the following notation: For any polynomial $U[X] \in H^{2}[X]$, we write $U_{F}$ for the polynomial obtained when we substitute the elements of the Hecke ring with their corresponding eigenvalues. We then have the following Definition.
\begin{defn}(\cite[Lemma 2.1]{gritsenko_zeta})
    The degree $6$ $L$-function attached to $F$ is defined as:
    \begin{equation*}
        Z_{F}^{(2)}(s) := \prod_{p \textup{ prime in }\mathcal{O}_K}(1+p^{k-s-2})^{-2}Q_{p,F}^{(2)}(p^{-s})^{-1}\prod_{\substack{p \textup{ splits}\\ \textup{or }p=2}}Q_{p,F}^{(2)}(p^{-s})^{-1},
    \end{equation*}
    where the polynomials $Q_{p}^{(2)}(t) \in H^{2}[t]$ are defined as follows:
    \begin{itemize}
        \item $p$ inert in $\mathcal{O}_K$:
        \begin{equation*}
            Q_p^{(2)}(t) := 1-T_pt + (pT_{1,p}+p(p^3+p^2-p+1)\Delta_p)t^2-p^4\Delta_pT_pt^3+p^8\Delta_p^2t^4.
        \end{equation*}
        \item $p = 2$:
        \begin{multline*}
            Q_2^{(2)}(t) := 1 - (T_2-3\Delta_{i+1})t+(2T_{1+i}^2-8\Delta_{1+i}(T_2 + \Delta_{1+i}))t^2-\\-(4\Delta_{1+i})^2(T_2-3\Delta_{1+i})t^3+(4\Delta_{1+i})^4t^4.
        \end{multline*}
        \item $p = \pi \overline{\pi}$, with $(\pi)\neq (\overline{\pi})$ splits in $\mathcal{O}_K$:
        \begin{multline*}
            Q_p^{(2)}(t) := 1 - T_pt+(pT_{\pi}T_{\overline{\pi}}-p^4\Delta_p)t^2-(p^3(T_{\pi}^2\Delta_{\overline{\pi}}+T_{\overline{\pi}}^2\Delta_{\pi})-2p^4\Delta_pT_p)t^3+\\+p^4\Delta_p(pT_{\pi}T_{\overline{\pi}}-p^4\Delta_p)t^4-p^8\Delta_p^2T_pt^5+p^{12}\Delta_p^3t^6.
        \end{multline*}
    \end{itemize}
    This is well-defined for $\textup{Re}(s)$ large enough (cf. \cite[p. 2546]{gritsenko_zeta}).
\end{defn}
We now define the \textbf{twisted degree $6$ $L$-function} attached to $F$ as follows. Let $N \geq 1$ be an integer. Once and for all, we fix a \textbf{Dirichlet character} $\chi: (\mathcal{O}_K/N\mathcal{O}_K)^{\times} \xrightarrow{} \mathbb{C}^{\times}$.
\begin{defn}
Let $F \in S_2^{k}$ be an eigenform for $H^2$. For $\textup{Re}(s)$ large enough, we define
\begin{equation*}
    Z_{F}^{(2)}(s, \chi) := \prod_{\textup{\textit{p} prime in }\mathcal{O}_K}(1+\chi(p)p^{k-2-s})^{-2}Q_{p, F}^{(2)}(\chi(p)p^{-s})^{-1}\prod_{\substack{p \textup{ splits}\\ \textup{or }p=2 }}Q_{p,F}^{(2)}(\chi(p)p^{-s})^{-1}.
\end{equation*}
\end{defn}
\section{Main Theorem}
In this Section, we prove the meromorphic continuation to $\mathbb{C}$ and functional equation for the twisted degree six $L$-function. We first make some definitions.\\

Let $\theta : (\mathbb{Z}/4\mathbb{Z})^{\times} \longrightarrow \mathbb{C}^{\times}$ denote the character with $\theta(3)=-1$ (equivalently the non-trivial quadratic character attached to the extension $K/\mathbb{Q}$).\\

For any Dirichlet character $\psi : (\mathbb{Z}/m\mathbb{Z})^{\times} \longrightarrow \mathbb{C}^{\times}$, where $m \geq 1$ an integer, we define
\begin{equation*}
    L(s, \psi) := \sum_{n=1}^{\infty}\psi(n)n^{-s} = \prod_{p \textup{ prime}}(1-\psi(p)p^{-s})^{-1}, \textup{ Re}(s)>1.
\end{equation*}

For a Dirichlet character $\omega: (\mathcal{O}_K/M\mathcal{O}_K)^{\times} \longrightarrow \mathbb{C}^{\times}$, where $M \geq 1$ an integer, we define 
\begin{equation*}
    \zeta_K(s, \omega):= \frac{1}{4}\sum_{0 \neq \lambda \in \mathcal{O}_K}\omega(\lambda)N(\lambda)^{-s}, \textup{ Re}(s)>1,
\end{equation*}
where $N(\lambda):=\lambda\overline{\lambda}$ denotes the \textbf{norm} of $\lambda$. We note that we also consider $L(s, \omega)$ by implicitly considering the restriction of $\omega$ to $\mathbb{Z}$.\\

For any $m \geq 1$, let $T(m)$ denote the standard Hecke element of $\textup{SL}_2(\mathbb{Z})$. We then define $T_{\pm}(m) := j_{\pm}(T(m))$, where $j_{\pm}$ are the embeddings of equation $\eqref{+- embeddings}$. Let also $\widetilde{\chi}$ denote the character given by $\widetilde{\chi}(\alpha) := \chi(\overline{\alpha})$. We then have the following Proposition.
\begin{proposition}\label{theorem1}
Let $F$ be a Hermitian cusp form of degree 2 and weight $k$ for the group $\Gamma_2$, which is an eigenfunction for all Hecke operators of the ring $H(\Gamma_2, S^2)$. Assume $\{\phi_m\}_{m=1}^{\infty}$ are its Fourier-Jacobi coefficients. Then, for $\textup{Re}(s)$ large enough, we have
\begin{multline*}
    L^2(s-k+2, \theta\chi)L(2s-2k+4, \chi^2)\zeta_{K}( s-k+3, \chi\widetilde{\chi})\sum_{m =1}^{\infty}\phi_m|_{k}T_{+}(m)\chi(m)m^{-s} =\\ =\zeta_{K}(s-k+2, \chi\widetilde{\chi})\phi_1Z_{F}^{(2)}(s,\chi).
\end{multline*}
If $N=1$ and $\chi$ is principal, this coincides with the result of Gritsenko in \cite[Theorem 4.1]{gritsenko}.
\end{proposition}
\begin{proof}
We will use \cite[Proposition 4.2]{gritsenko}, which states that in the ring of formal power series $J_{k,1}[[t]]$, we have:
\begin{multline*}
    Q_{p,F}^{(2)}(t)\sum_{\delta \geq 0}(\phi_{mp^{\delta}}|_{k}T_{+}(mp^{\delta}))t^{\delta} = \\=\begin{cases}
      (1-p^{2k-6}t^2)\phi_m|_{k}T_{+}(m) & \text{$p$ prime in $\mathcal{O}_K$}\\
      (1-p^{k-3}t)^2(1-p^{2k-4}t^2)\phi_m|_{k}T_{+}(m) & \text{$p$ splits in $\mathcal{O}_K$}\\
      (1-2^{k-3}t)(1+2^{k-2}t)\phi_m|_{k}T_{+}(m) & \text{$p=2$}
    \end{cases},     
\end{multline*}
if $\gcd(m,p)=1$. Now, if $\gcd(m,p)=1$, by setting $t \longmapsto \chi(p)t$ and multiplying both sides with $\chi(m)$, we obtain:
\begin{multline}\label{expression gritsenko}
    Q_{p,F}^{(2)}(\chi(p)t)\sum_{\delta \geq 0}(\phi_{mp^{\delta}}|_{k}T_{+}(mp^{\delta}))\chi(mp^{\delta})t^{\delta} = \\=\begin{cases}
      (1-p^{2k-6}\chi^2(p)t^2)\phi_m|_{k}T_{+}(m)\chi(m) & \text{$p$ prime in $\mathcal{O}_K$}\\
      (1-p^{k-3}\chi(p)t)^2(1-p^{2k-4}\chi^2(p)t^2)\phi_m|_{k}T_{+}(m)\chi(m) & \text{$p$ splits in $\mathcal{O}_K$}\\
      (1-2^{k-3}\chi(p)t)(1+2^{k-2}\chi(p)t)\phi_m|_{k}T_{+}(m)\chi(m) & \text{$p=2$}
    \end{cases},     
\end{multline}
using also that $\chi$ is completely multiplicative.
By then applying this successively, we have that the expression
\begin{equation*}
    \sum_{m=1}^{\infty}\phi_m|_{k}T_{+}(m)\chi(m)m^{-s}
\end{equation*}
is equal to $\phi_1$ times an Euler product with $p$-factor
\begin{equation*}
    \begin{cases}
      (1-\chi^2(p)p^{-2s+2k-6})Q_{p,F}^{(2)}(\chi(p)p^{-s})^{-1} & \text{$p$ prime in $\mathcal{O}_K$}\\
      (1-\chi(p)p^{-s+k-3})^2(1-\chi^{2}(p)p^{-2s+2k-4})Q_{p,F}^{(2)}(\chi(p)p^{-s})^{-1} & \text{$p$ splits in $\mathcal{O}_K$}\\
      (1-\chi(2)2^{-s+k-3})(1+\chi(2)2^{-s+k-2})Q_{2,F}^{(2)}(\chi(2)2^{-s})^{-1}  & \text{$p=2$}
    \end{cases}.
\end{equation*}
Indeed, by fixing a single prime $p$ and assuming, for example, that it is inert, we can write
\begin{multline*}
    \sum_{m=1}^{\infty}\phi_m |_k T_{+}(m)\chi(m)m^{-s}=\sum_{\substack{m=1\\ (m,p)=1}}^{\infty}  \sum_{\delta \geq 0} \phi_{mp^\delta} |_k T_{+}(mp^{\delta}) \chi(mp^{\delta})(mp^{\delta})^{-s} = \\=(1-\chi^2(p)p^{-2s+2k-6})Q_{p,F}^{(2)}(\chi(p)p^{-s})^{-1}\sum_{\substack{m=1\\ (m,p)=1}}^{\infty}\phi_m |_k T_{+}(m)\chi(m)m^{-s},
\end{multline*}
using \eqref{expression gritsenko}. By doing this repeatedly for all primes, we obtain the Euler product above. The result of the Proposition now follows by comparing the local factors for each rational prime $p$, because of the following Euler product expressions:
\begin{equation*}
        L(s, \theta \chi)  = \prod_{p \textup{ prime in }\mathcal{O}_K}(1+\chi(p)p^{-s})^{-1}\prod_{2 \neq p \textup{ splits}}(1-\chi(p)p^{-s})^{-1}.
\end{equation*}
\begin{equation*}
    L(s,\chi^2) = \prod_{p}(1-\chi^2(p)p^{-s})^{-1}.
\end{equation*}
\begin{equation*}
    \zeta_K(s, \chi\widetilde{\chi}) = (1-\chi(2)2^{-s})^{-1}\prod_{p \textup{ prime in }\mathcal{O}_K}(1-\chi^2(p)p^{-2s})^{-1}\prod_{2 \neq p \textup{ splits}}(1-\chi(p)p^{-s})^{-2}.
\end{equation*}
\end{proof}
In order to now deduce the analytic properties of $Z_F(s, \chi)$, we consider a specific Dirichlet series, as Das and Jha did in \cite{das_jha}. For any $F, G \in S_2^{k}$, with Fourier-Jacobi expansions
\begin{equation*}
    F(Z) = \sum_{m=1}^{\infty}\phi_m(\tau, z_1, z_2)e(m\omega), \textup{ }G(Z) = \sum_{m=1}^{\infty}\psi_m(\tau, z_1, z_2)e(m\omega),
\end{equation*}
we define a \textbf{twisted Dirichlet series} attached to them via
\begin{equation}\label{twisted dirichlet}
    D_{F,G}(s,\chi):= \zeta(2s-2k+4)\zeta_K(s-k+3,\chi\widetilde{\chi})\sum_{m=1}^{\infty}\chi(m)\langle\phi_m, \psi_m\rangle m^{-s},
\end{equation}
where $\zeta(s)$ denotes the usual Riemann zeta function. This series converges absolutely and uniformly in compact subsets for $\textup{Re}(s)>k+1$ (cf. \cite[Section 4]{das_jha}) and also has good analytic properties (\cite[Theorem 4.1]{das_jha}).\\

We also define the \textbf{completion} of $Z_{F}(s, \chi)$ and the \textbf{completed twisted Dirichlet series} via
\begin{multline}\label{completed zeta}
    Z_{F}^{*}(s,\chi) := N^{-4s}(4\pi^3)^{-s}\Gamma(s)\Gamma(s-k+3)\Gamma(s-k+2)\times\\\times\frac{\zeta(2s-2k+4)\zeta_{K}(s-k+2, \chi\widetilde{\chi})}{L^2(s-k+2, \theta\chi)L(2s-2k+4, \chi^2)}Z_{F}^{(2)}(s,\chi)
\end{multline}
and
\begin{multline}\label{completed dirichlet}
    D_{F, G}^{*}(s,\chi) := N^{-4s}(4\pi^3)^{-s}\Gamma(s)\Gamma(s-k+3)\Gamma(s-k+2)\zeta(2s-2k+4)\times\\\times\zeta_{K}(s-k+3, \chi\widetilde{\chi})\sum_{m =1}^{\infty}\chi(m)\langle \phi_m, \psi_m \rangle m^{-s}.
\end{multline}
These are well-defined for $\textup{Re}(s)$ large enough. Here, $\displaystyle{\Gamma(s):=\int_{0}^{\infty}t^{s-1}e^{-t}\textup{d}t}$ denotes the usual (analytically continued) Gamma function.\\

We then have the following main Theorem.
\begin{theorem}\label{theorem2}
For a Hermitian cusp form $F\in S_2^k$, which is an eigenfunction for all Hecke operators of the ring $H(\Gamma_2, S^2)$ and has first Fourier-Jacobi coefficient $\phi_1 \not\equiv 0$, the function $Z_{F}^{*}(s,\chi)$ can be holomorphically continued to the whole complex plane, except possibly for some simple poles at $s=k, k-1, k-2, k-3$, which are only possible when $\chi\widetilde{\chi}$ is principal. In either case, it satisfies the functional equation
\begin{equation*}
    Z_{F}^{*}(2k-3-s,\chi) = Z_{F}^{*}(s,\chi).
\end{equation*}
\end{theorem}
\begin{proof}
Let $\psi_1$ be a Jacobi cusp form of weight $k$ and index $1$. From \cite[Proposition 5.1]{gritsenko} we have that the adjoint operator of $T_{+}(m)$ with respect to the inner product of Jacobi forms is $m^{3-k}T_{-}(m)$. By now taking the inner product of the result of Proposition \ref{theorem1} with $\psi_1$, we obtain:
\begin{multline*}
     L^2(s-k+2,\theta\chi)L(2s-2k+4,\chi^2)\zeta_{K}(s-k+3,\chi\widetilde{\chi})\sum_{m=1}^{\infty}\langle \phi_m|_kT_{+}(m), \psi_1 \rangle \chi(m)m^{-s}
     \\= \zeta_{K}(s-k+2,\chi\widetilde{\chi})\langle \phi_1, 
     \psi_1 \rangle Z_{F}^{(2)}(s,\chi),
\end{multline*}
which is equivalent to
\begin{multline*}
     L^2(s-k+2,\theta\chi)L(2s-2k+4,\chi^2)\zeta_{K}(s-k+3,\chi\widetilde{\chi})\sum_{m=1}^{\infty}\langle \phi_m, \psi_1|_{k}m^{3-k}T_{-}(m) \rangle \chi(m)m^{-s}
     \\= \zeta_{K}(s-k+2,\chi\widetilde{\chi})\langle \phi_1, 
     \psi_1 \rangle Z_{F}^{(2)}(s,\chi),
\end{multline*}
after passing to the adjoint. But from \cite[Proposition 5.2]{gritsenko}, we have that the function
\begin{equation*}
    G_{\psi_1}\left(\begin{pmatrix} \tau & z_1\\ z_2 & \omega \end{pmatrix}\right) := \sum_{m \geq 1}m^{3-k}(\psi_1|_k T_{-}(m))(\tau, z_1, z_2)e(m\omega) 
\end{equation*}
is a Hermitian cusp form of degree $2$ and weight $k$ (and actually belongs in the Maass space, see \cite[Appendix]{gritsenko}). In particular, if $\{\psi_m\}_{m=1}^{\infty}$ are the Fourier-Jacobi coefficients of $G_{\psi_1}$, we have $\psi_m = m^{3-k}\psi_1 |_{k}T_{-}(m)$ for all $m \geq 1$. Therefore, after multiplying with $\zeta(2s-2k+4)$, the above can be written as
\begin{multline*}
    L^2(s-k+2,\theta\chi)L(2s-2k+4,\chi^2)D_{F,G_{\psi_1}}(s, \chi) =\\= \zeta(2s-2k+4)\zeta_{K}(s-k+2,\chi\widetilde{\chi})\langle \phi_1, 
     \psi_1 \rangle Z_{F}^{(2)}(s,\chi),
\end{multline*}
using \eqref{twisted dirichlet}. This is, in turn, equivalent to
\begin{equation*}
    D_{F, G_{\psi_1}}^{*}(s,\chi) = \langle \phi_1, \psi_1 \rangle Z_F^{*}(s,\chi),
\end{equation*}
using the expressions \eqref{completed zeta} and \eqref{completed dirichlet}. Therefore, because of the assumption $\phi_1 \not \equiv 0$, the analytic properties of $Z_{F}^{*}(s, \chi)$ coincide with those of $D_{F,G_{\psi_1}}^{*}(s,\chi)$, which are exactly the ones mentioned in the statement of the Theorem, using the result of Das and Jha in \cite[Theorem 4.1]{das_jha}.
\end{proof}
\renewcommand{\abstractname}{Acknowledgements}
\begin{abstract}
This work was supported by the Additional Funding Programme for Mathematical Sciences, delivered by EPSRC (EP/V521917/1) and the Heilbronn Institute for Mathematical Research as well as by a scholarship by Onassis Foundation. The author would also like to thank Prof. Thanasis Bouganis for his continuous guidance and support.
\end{abstract}
\printbibliography

@article{andrianov2,
year = {1974},
publisher = {},
volume = {\textbf{29}},
number = {\textbf{3}},
pages = {45-116},
author = {A. N. Andrianov},
title = {{Euler Products corresponding to Siegel modular forms of genus 2}},
journal = {Russian Mathematical Surveys},
}

@article{gritsenko_zeta,
author = {Gritsenko, V. A.},
title = {Zeta function of degree six for Hermite modular forms of genus two},
journal = {J Math Sci},
volume = {\textbf{43}},
year = {1988},
pages = {2540–2553}
}

@article{gritsenko,
    author = "Gritsenko, V. A.",
    title = "Jacobi functions and Euler products for Hermitian modular forms.",
    journal = "J Math Sci",
    volume = "\textbf{62}",
    pages = {2883-2914},
    year = "1992"
}

@article{gritsenko3,
author = {Gritsenko, V. A.},
title = {{Parabolic extensions of the Hecke ring of the general linear group. II}},
volume = {\textbf{62}},
journal = {J Math Sci},
pages = {2869-2882},
year = {1992}
}

@book {Krieg_Book,
    AUTHOR = {Krieg, A.},
     TITLE = {Modular forms on half-spaces of quaternions},
    SERIES = {Lecture Notes in Mathematics},
    VOLUME = {\textbf{1143}},
 PUBLISHER = {Springer-Verlag, Berlin},
      YEAR = {1985},
     PAGES = {xiii+203},
      ISBN = {3-540-15679-8},
   MRCLASS = {11F55 (11F46 11H55)},
  MRNUMBER = {807947},
MRREVIEWER = {K.-B.\ Gundlach},
}

@article{kohnen_skoruppa,
    author = {Kohnen, W. and Skoruppa, N. P.},
    title = {{A certain Dirichlet series attached to Siegel modular forms of degree two.}},
    journal = {Invent Math},
    volume = {\textbf{95}},
    pages = {541–558},
    year = {1989},
}

@article{kohnen_krieg_sengupta,
    author = "Kohnen, W. and Sengupta, J. and Krieg, A.",
    title = "Characteristic Twists of a Dirichlet series for Siegel Cusp Forms." ,
    journal = "manuscripta mathematica",
    volume = "\textbf{87}",
    pages = "489-500",
    year = "1995",
}

@article{manickam,
title = {On the first Fourier-Jacobi coefficient of Siegel modular forms of degree two},
journal = {Journal of Number Theory},
volume = {\textbf{219}},
pages = {404-411},
year = {2021},
author = {M. Manickam}
}

@article{raghavan,
author = {Raghavan, S. and Sengupta, J.},
journal = {Acta Arithmetica},
number = {\textbf{2}},
pages = {181-201},
title = {A Dirichlet series for Hermitian modular forms of degree 2},
volume = {\textbf{58}},
year = {1991},
}

@article{das_jha,
   author = "Das, S. and Jha, A.K.",
    title = "Analytic properties of twisted real-analytic Hermitian Klingen type Eisenstein series and applications ." ,
    journal = "Abh. Math. Semin. Univ. Hambg. 89",
    pages = "105–116",
    year = "2019",
    DOI = "https://doi.org/10.1007/s12188-019-00206-7"
}

@article{kohnen,
  title="{On characteristic twists of certain Dirichlet series}",
  author={W. Kohnen},
  journal={Memoirs of the Faculty of Science, Kyushu University. Series A, Mathematics},
  volume={\textbf{47}},
  number={\textbf{1}},
  pages={103-117},
  year={1993},
}
\end{document}